\documentclass[11pt]{article}

\usepackage{fullpage}

\usepackage{amsmath, amssymb, amsthm}

\usepackage[utf8]{inputenc}
\usepackage[T1]{fontenc}

\usepackage{parskip}

\newtheorem{theorem}{Theorem}[section]

\title{Note on Finite-Automata Bernoulli Factories\\for Rational Functions}
\date{\today}
\author{Renato Paes Leme \\
Google Research \and Jon Schneider\\
Google Research}

\begin{document}

\maketitle

\begin{abstract}
Mossel and Peres (2005) established a comprehensive framework for designing Bernoulli factories. Notably, they demonstrated that a single-variable function admits a finite-automata Bernoulli factory if and only if it is a rational function. Their Theorem 2.9 claims an extension of this result to multivariable functions, but it contains a subtle technical oversight in the application of Pólya's Theorem. We provide a direct counterexample: a rational function in three variables that admits a general Bernoulli factory but cannot be implemented by a finite-automata Bernoulli factory.

\end{abstract}

\section{Introduction}

The Bernoulli factory problem asks whether one can simulate a target probability distribution using samples from an unknown source distribution, provided the source distribution belongs to a known family. A foundational result in this area is Keane and O'Brien's theorem \cite{keane1994}, which characterizes the functions implementable by Bernoulli factories.

Mossel and Peres \cite{mossel2005} generalized and streamlined these ideas, investigating Bernoulli factories for functions $f: \Delta_s \to (0,1)$, where $\Delta_s = \{ (p_1, \dots, p_{s+1}) \in (0,1)^{s+1} : \sum_{i=1}^{s+1} p_i = 1 \}$ denotes the open $s$-dimensional simplex. They consider both general Bernoulli factories and finite-automata Bernoulli factories (also called simulation via \textit{blocks}).
They show that a function $f: \Delta_s \to (0,1)$ admits a finite-automata Bernoulli factory if it can be represented as a ratio of two homogeneous polynomials with non-negative coefficients (often referred to as Bernstein polynomials). In Section 2 of their paper, Mossel and Peres present the following algebraic characterization:

\begin{quote}
\textbf{Theorem 2.9 (Mossel and Peres \cite{mossel2005}).} \textit{Any rational function $f: \Delta_s \to (0,1)$ can be simulated via blocks.}
\end{quote}

The overarching idea behind their proof is elegant: given a rational function $f = A/B$ mapping the simplex $\Delta_s$ to $(0,1)$, one can invoke Pólya's Theorem \cite{polya1928} to multiply the numerator and denominator by a suitable power of $(p_1 + \dots + p_{s+1})$, thereby clearing any negative coefficients in $B$.

This proof strategy is perfectly fine for the single-variable case ($s=1$, representing a single unknown bias $p_1$ and $p_2 = 1-p_1$). They state without proof that their ideas can be extended to multiple variables ($s \ge 2$) using the same proof. However, a subtle technical detail in applying Pólya's Theorem in higher dimensions was overlooked.

Specifically, there is a crucial distinction between the open simplex $\Delta_s$ and the closed simplex $\bar{\Delta}_s$. Pólya's Theorem guarantees that if a homogeneous polynomial $B(\mathbf{p})$ is strictly positive on the \textit{closed} simplex $\bar{\Delta}_s$, then for sufficiently large $k$, the polynomial $B(\mathbf{p})(p_1 + \dots + p_{s+1})^k$ has strictly non-negative coefficients. When $f$ is bounded away from $0$ and $1$ on the \textit{open} simplex $\Delta_s$, it does not preclude the coprime denominator $B(\mathbf{p})$ from vanishing at points on the boundary $\partial \Delta_s$. 

In the single-variable case ($s=1$), this boundary vanishing does not present an obstacle. In the polynomial ring $\mathbb{R}[p_1, p_2]$, any zeros on the boundary (viz., at $(1,0)$ or $(0,1)$) correspond to linear factors $p_1$ or $p_2$. Since $f = A/B$ is bounded away from $0$ and $1$, any boundary zeros of $B$ must be shared by $A$ and can be factored out and canceled. However, in multivariable rings ($s \ge 2$), coprime polynomials can simultaneously vanish on the boundary without sharing common factors. Consequently, the denominator can vanish on $\partial \Delta_s$ in a way that cannot be factored out.

In this note, we make this observation explicit by constructing a rational function on the 2-simplex $\Delta_2$ that admits a Bernoulli factory but cannot be implemented by a finite-automata Bernoulli factory, thereby providing a direct counterexample to Theorem 2.9 in higher dimensions. Geometrically, the counterexample is related to a cusp singularity on a half-plane in algebraic geometry.

\section{Setup}

Let $\Delta_s = \{ \mathbf{p} = (p_1, \dots, p_{s+1}) \in (0,1)^{s+1} : \sum_{i=1}^{s+1} p_i = 1 \}$ denote the open $s$-dimensional simplex, and let $\bar{\Delta}_s$ denote its closure. A polynomial $P(\mathbf{p}) \in \mathbb{R}[p_1, \dots, p_{s+1}]$ is called a \textit{Bernstein polynomial} on the simplex if it is homogeneous and all its coefficients are non-negative; that is, $P(\mathbf{p}) = \sum_{\alpha} c_\alpha p_1^{\alpha_1} \cdots p_{s+1}^{\alpha_{s+1}}$ with $c_\alpha \ge 0$ and $\sum_{i=1}^{s+1} \alpha_i = d$.

The necessary and sufficient conditions for a function to admit a general Bernoulli factory were originally established for the single-variable case by Keane and O'Brien \cite{keane1994}. Morina \cite{morina2021} (see also \cite{leme2022}) generalized this characterization to the multivariate simplex, establishing the following foundational criteria:

\begin{theorem}[Morina \cite{morina2021}]
\label{thm:morina}
A function $f: \Delta_s \to (0,1)$ admits a Bernoulli factory if and only if $f$ is continuous and there exist constants $c > 0$ and $k \ge 1$ such that for all $\mathbf{p} \in \Delta_s$:
\[
\min(f(\mathbf{p}), 1 - f(\mathbf{p})) \ge c \prod_{i=1}^{s+1} p_i^k
\]
\end{theorem}

As established in \cite{mossel2005}, a rational function $f: \Delta_s \to (0,1)$ is implementable by a finite-automata Bernoulli factory if and only if there exist Bernstein polynomials $A(\mathbf{p})$ and $B(\mathbf{p})$ such that $f(\mathbf{p}) = A(\mathbf{p})/B(\mathbf{p})$ for all $\mathbf{p} \in \Delta_s$.

We recall that the multivariate polynomial ring $\mathbb{R}[p_1, \dots, p_{s+1}]$ is a Unique Factorization Domain (UFD) and an integral domain. Consequently, if $f = P/Q$ where $P$ and $Q$ are coprime, any alternative rational representation $A/B = P/Q$ must satisfy $A = P \cdot R$ and $B = Q \cdot R$ for some non-zero polynomial $R \in \mathbb{R}[p_1, \dots, p_{s+1}]$.

\section{Counterexample}

We now present the main result of this note: an explicit counterexample to Theorem 2.9 in \cite{mossel2005} for multivariable rational functions.

\begin{theorem}
\label{thm:counterexample}
There exists a rational function $g: \Delta_2 \to (0,1)$ that is polynomially bounded away from $0$ and $1$ (and hence admits a Bernoulli factory) but cannot be implemented by a finite-automata Bernoulli factory.
\end{theorem}

\begin{proof}
Consider the homogeneous rational function $g: \Delta_2 \to (0,1)$ defined by:
\[
g(p_1, p_2, p_3) = \frac{p_1^3 p_2}{p_3^2 (p_1 - p_2)^2 + p_1^3 (p_1 + p_2 + p_3)}
\]
Let $P(\mathbf{p}) = p_1^3 p_2$ and $Q(\mathbf{p}) = p_3^2 (p_1 - p_2)^2 + p_1^3 (p_1 + p_2 + p_3)$. Note that on the simplex $\Delta_2$, $p_1 + p_2 + p_3 = 1$, so $Q(\mathbf{p})$ can be viewed as $p_3^2 (p_1 - p_2)^2 + p_1^3$.

\textbf{Part 1: Existence of a Bernoulli Factory.}
We first verify that $g(\mathbf{p})$ is strictly bounded between $0$ and $1$ on the open simplex $\Delta_2$ and satisfies the requisite polynomial boundedness condition near the boundary.
For any $\mathbf{p} \in \Delta_2$, $p_1, p_2, p_3 > 0$. Thus, $P(\mathbf{p}) = p_1^3 p_2 > 0$.
For the denominator, since $(p_1-p_2)^2 \ge 0$ and $p_1^3 > 0$, we have $Q(\mathbf{p}) > 0$.
Furthermore, examining the difference $Q(\mathbf{p}) - P(\mathbf{p})$, we have:
\[
Q(\mathbf{p}) - P(\mathbf{p}) = p_3^2 (p_1 - p_2)^2 + p_1^4 + p_1^3 p_2 + p_1^3 p_3 - p_1^3 p_2 = p_3^2 (p_1 - p_2)^2 + p_1^4 + p_1^3 p_3
\]
Since $p_1, p_3 > 0$, this difference is strictly positive. Thus, $0 < g(\mathbf{p}) < 1$ for all $\mathbf{p} \in \Delta_2$.

To verify polynomial boundedness near the boundary, observe that $Q(\mathbf{p})$ is a continuous function on the compact set $\bar{\Delta}_2$. Since each term in $Q(\mathbf{p})$ is bounded by $1$, $Q(\mathbf{p}) \le 2$ for all $\mathbf{p} \in \bar{\Delta}_2$. Thus, $1/Q(\mathbf{p}) \ge 1/2$, which implies $g(\mathbf{p}) = \frac{p_1^3 p_2}{Q(\mathbf{p})} \ge \frac{1}{2} p_1^3 p_2$. Since $p_i < 1$, this is bounded below by a positive constant times a monomial in $p_1, p_2, p_3$. Similarly, for $1 - g(\mathbf{p})$, we have $1 - g(\mathbf{p}) = \frac{Q(\mathbf{p}) - P(\mathbf{p})}{Q(\mathbf{p})} \ge \frac{1}{2} (p_1^4 + p_1^3 p_3) \ge \frac{1}{2} p_1^4$, which is also bounded below by a positive polynomial in the simplex variables. By Theorem \ref{thm:morina}, $g$ admits a Bernoulli factory.

\textbf{Part 2: Impossibility of Implementation by a Finite-Automata Bernoulli Factory.}
Suppose for contradiction that $g$ can be implemented by a finite-automata Bernoulli factory. Then there exist Bernstein polynomials $A(\mathbf{p})$ and $B(\mathbf{p})$ (homogeneous with non-negative coefficients) such that $g(\mathbf{p}) = A(\mathbf{p})/B(\mathbf{p})$.

We first establish that $P(\mathbf{p})$ and $Q(\mathbf{p})$ are coprime in $\mathbb{R}[p_1, p_2, p_3]$. The only irreducible factors of $P(\mathbf{p})$ are $p_1$ and $p_2$. Evaluating $Q(\mathbf{p})$ at $p_1 = 0$ yields $Q(0, p_2, p_3) = p_3^2 p_2^2 \neq 0$, so $p_1$ does not divide $Q$. Evaluating $Q(\mathbf{p})$ at $p_2 = 0$ yields $Q(p_1, 0, p_3) = p_3^2 p_1^2 + p_1^3(p_1 + p_3) \neq 0$, so $p_2$ does not divide $Q$. Thus, $P$ and $Q$ share no common non-constant factors.

Since $\mathbb{R}[p_1, p_2, p_3]$ is a UFD and $A/B = P/Q$, there must exist a non-zero homogeneous polynomial $R(\mathbf{p})$ such that $B(\mathbf{p}) = Q(\mathbf{p}) R(\mathbf{p})$.

To analyze this relation, we dehomogenize at the vertex $(0,0,1)$ by setting $p_3 = 1$. Let $b(p_1, p_2) = B(p_1, p_2, 1)$, $q(p_1, p_2) = Q(p_1, p_2, 1)$, and $r(p_1, p_2) = R(p_1, p_2, 1)$. Since $B(\mathbf{p})$ is a Bernstein polynomial, $b(p_1, p_2)$ must have strictly non-negative coefficients.

Substituting $p_3 = 1$ into $Q$, we obtain $q(p_1, p_2) = (p_1 - p_2)^2 + p_1^3 (p_1 + p_2 + 1)$. The lowest homogeneous part of $q(p_1, p_2)$ is $q_2(p_1, p_2) = (p_1 - p_2)^2$, which has degree $2$.

Let $r_k(p_1, p_2)$ be the lowest non-zero homogeneous part of $r(p_1, p_2)$, where $k \ge 0$ is its degree. The lowest non-zero homogeneous part of the product $b(p_1, p_2) = q(p_1, p_2) r(p_1, p_2)$ has degree $k+2$, and is given exactly by the product of the lowest homogeneous parts:
\[
b_{k+2}(p_1, p_2) = q_2(p_1, p_2) r_k(p_1, p_2) = (p_1 - p_2)^2 r_k(p_1, p_2)
\]
Since $B(\mathbf{p})$ is a Bernstein polynomial, the dehomogenized polynomial $b(p_1, p_2)$ has non-negative coefficients. Because no cancellation can occur between terms of different degrees, its lowest homogeneous part $b_{k+2}(p_1, p_2)$ must also have strictly non-negative coefficients.

Evaluating $b_{k+2}(p_1, p_2)$ along the ray $p_1 = p_2 = 1$ yields $b_{k+2}(1, 1) = (1 - 1)^2 r_k(1, 1) = 0$. For a polynomial with non-negative coefficients, evaluating at $(1,1)$ yields the exact sum of its coefficients. Since this sum is $0$ and all coefficients are non-negative, every coefficient of $b_{k+2}(p_1, p_2)$ must be $0$. Thus, $b_{k+2}(p_1, p_2)$ is the identically zero polynomial.

However, the polynomial ring $\mathbb{R}[p_1, p_2]$ is an integral domain. Since neither $(p_1 - p_2)^2$ nor $r_k(p_1, p_2)$ is the zero polynomial, their product $b_{k+2}(p_1, p_2)$ cannot be zero. This is a contradiction.

Therefore, no such Bernstein polynomial $B(\mathbf{p})$ exists, proving that $g(\mathbf{p})$ cannot be implemented by a finite-automata Bernoulli factory.
\end{proof}

\paragraph{Tool Use Disclosure} The authors used the DeepThink mode of Gemini 3.1 Pro to derive the counterexample and to refine the presentation.



\begin{thebibliography}{9}

\bibitem{keane1994}
M.~S. Keane and O.~B. O'Brien, ``A Bernoulli factory,'' \emph{ACM Transactions on Modeling and Computer Simulation (TOMACS)}, vol.~4, no.~2, pp.~213--219, 1994.

\bibitem{leme2022}
R.~Paes Leme and J.~Schneider, ``Multiparameter Bernoulli factories,'' \emph{The Annals of Applied Probability}, vol.~33, no.~5, pp.~3987--4007, 2023.

\bibitem{morina2021}
G.~Morina, \emph{Extending the Bernoulli Factory to a Dice Enterprise}, Ph.D. dissertation, University of Warwick, 2021.

\bibitem{mossel2005}
E.~Mossel and Y.~Peres, ``New coins from old: computing with unknown bias,'' \emph{Combinatorica}, vol.~25, no.~6, pp.~707--724, 2005.

\bibitem{polya1928}
G.~Pólya, ``Über positive Darstellung von Polynomen,'' \emph{Vierteljahrsschrift der Naturforschenden Gesellschaft in Zürich}, vol.~73, pp.~141--145, 1928.

\end{thebibliography}
\end{document}